\documentclass[11pt,reqno]{amsart}
\usepackage{amsmath,amssymb,mathrsfs}
\usepackage{graphicx,cite}
\usepackage{color}
\usepackage{subfigure}
\usepackage{appendix}

\newtheorem{theorem}{Theorem}[section]

\newtheorem{lemma}[theorem]{Lemma}

\newtheorem{remark}[theorem]{Remark}

\newcommand{\E}{\mathbb{E}} 
\newcommand{\ii}{{\rm i}}
\numberwithin{equation}{section}

\allowdisplaybreaks[4]
\begin{document}

\title[stability for inverse random source problems]{Stability for inverse source problems of the stochastic Helmholtz equation with a white noise}

\author{Peijun Li}
\address{Department of Mathematics, Purdue University, West Lafayette, Indiana 47907, USA}
\email{lipeijun@math.purdue.edu}

\author{Ying Liang}
\address{Department of Mathematics, Purdue University, West Lafayette, Indiana 47907, USA}
\email{liang402@purdue.edu}

\thanks{The research is supported in part by the NSF grant  DMS-2208256.}



\maketitle
\begin{abstract}
This paper is concerned with the stability estimates for inverse source problems of the stochastic Helmholtz equation driven by white noise. The well-posedness is established for the direct source problems, which ensures the existence and uniqueness of solutions. The stability estimates are deduced for the inverse source problems, which aim to determine the strength of the random source. To enhance the stability of the inverse source problems, we incorporate a priori information regarding the regularity and support of the strength. In the case of homogeneous media, a H\"{o}lder stability estimate is established, providing a quantitative measure of the stability for reconstructing the source strength. For the more challenging scenario of inhomogeneous media, a logarithmic stability estimate is presented, capturing the intricate interactions between the source and the varying medium properties.
\end{abstract}

\section{Introduction}

Inverse source problems in wave propagation are of great importance in various scientific and engineering fields, as they allow us to reveal the underlying characteristics and properties of the sources that generate observed wave signals. These problems have wide-ranging implications across disciplines such as seismology, acoustics, non-destructive testing, medical imaging, and telecommunications. By successfully solving inverse source problems, we gain insights into the location, nature, and behavior of hidden or unknown sources, leading to improved understanding, prediction, and control of wave phenomena. Due to the existence of non-radiating sources, inverse source problems typically lack unique solutions when only boundary measurements at a fixed frequency are available \cite{bleistein1977nonuniqueness, devaney1982nonuniqueness, eller2009acoustic}. Furthermore, inverse source problems are unstable and their solutions are highly sensitive to small perturbations in the measurements, leading to significant variations in the reconstructed solutions. In order to tackle these challenges, researchers have made efforts to identify the least energy solution \cite{marengo1999inverse, dml2007siap}. By utilizing multi-frequency data, uniqueness can be ensured and improved stability can be achieved in solving inverse source problems \cite{BLT2010jde, cheng2016increasing, li2020stability}. We refer to \cite{BLLT15} for a comprehensive review of solving inverse scattering problems by using multiple frequency data. 

While significant research has been conducted on deterministic counterparts, there is a growing interest in addressing inverse source problems with uncertainties, which arise from unpredictable environmental conditions, incomplete system information, and random measurement noise. Inverse random source problems have found extensive applications spanning various domains. In environmental monitoring, they are employed to study noise and pollution sources. In the field of seismology and geophysics, these problems are utilized to analyze seismic events such as earthquakes and underground explosions. Furthermore, in oceanographic studies, inverse random source problems are employed to investigate phenomena like turbulence and wave breaking. Unlike deterministic inverse source problems, where the focus is on precisely reconstructing the properties of fixed sources, inverse random source problems involve sources that are inherently random or uncertain. These sources are modeled as random variables or stochastic processes, and the problem addresses the challenges associated with uncertain and random sources by focusing on probabilistic estimation and statistical descriptions of their characteristics.

The inverse random source problem was first investigated in \cite{devaney1979jmp}, where specific instances were examined to determine the auto-correlation of random sources uniquely. In \cite{BCLZ14} and \cite{li2011ip}, a computational framework was established for the inverse random source problem of the one-dimensional Helmholtz equation in homogeneous and inhomogeneous media, respectively, where random sources were modeled as white noise. The approach was extended to address higher-dimensional problems in \cite{BCL16, BCL17, li2017inverse} for the acoustic and elastic wave equations. In recent studies \cite{LHL-cpde2020, li2022inverse, LW21}, uniqueness was discussed for the inverse random source problems, considering the source as a generalized microlocally isotropic Gaussian random field. In \cite{hohage2020ip}, the authors discussed the uniqueness of an inverse source problem in the context of experimental  aeroacoustics by using correlations of measured wave field. While the uniqueness of the inverse random source problems has received considerable attention in the literature, stability results are notably rare. The work presented in \cite{li2017stability} explored the increasing stability of the inverse random source problem associated with the one-dimensional Helmholtz equation in a homogeneous medium, where the source was driven by white noise. In this pioneering study, we aim to address the stability of inverse source problems for the three-dimensional stochastic Helmholtz equation in both homogeneous and inhomogeneous media. 
 
Specifically, we consider the Helmholtz equation with a random source
\begin{equation}\label{eqn:main}
\Delta u +k^2 (1+q) u = f,\quad x\in \mathbb{R}^3,
\end{equation}
where $k>0$ is the wave number, $q$ describes the relative electric permittivity of the medium and is assumed to be compactly supported in the unit ball $B_1=\{x\in\mathbb R^3: |x|<1\}$, $u$ describes the spatial distribution or variation of the wave field, the random source $f$ denotes the distribution of electric current density and is assumed to have a confined range within $B_1$. Moreover, we assume that $q\in C^{1,\alpha}(\overline{B_1}), 0<\alpha<1$ and $ \Im q \geq 0$, i.e., $q$ is allowed to be a complex-valued function. When $q\equiv 0$, i.e., the medium is homogeneous, the stochastic Helmholtz equation \eqref{eqn:main} reduces to
\begin{equation}\label{eqn:mainh}
\Delta u +k^2 u = f,\quad x\in \mathbb{R}^3. 
\end{equation}
Moreover, the random source $f$ is assumed to take the form
\[
 f(x) = \sqrt{\mu}(x) \dot{W}_x,
\]
where $\dot W_x$ is the spatial white noise with $W_x$ being the Brownian motion defined in the complete probability space $(\Omega, \mathcal F, \mathbb P)$, and $\mu\geq 0$ is called the strength of the random source satisfying $\mu\in H_0^s(B_1), s>3$. As usual, the Sommerfeld radiation condition is imposed on the wave field $u$, i.e., $u$ satisfies
\begin{equation}\label{eqn:rc}
\lim_{r\to\infty} r\left(\partial_r u -\ii  ku\right)=0,\quad r=|x|.
\end{equation}

Denote $B_R=\{x\in\mathbb R^3: |x|<R\}$ with the boundary $\partial B_R=\{x\in\mathbb R^3: |x|=R\}$, where $R>1$. By Lemma \ref{lem:boundary}, we have that $u$ is analytic in $\mathbb R^3\setminus \overline{B_1}$. Hence the Cauchy data $u$ and $\partial_\nu u$ are well-defined on $\partial B_R$, where $\nu$ is the outward normal vector on $\partial B_R$. Define the following correlation functions: 
\begin{align*}
F_1(x,y)&=\E[u(x) u(y)],\\
F_2(x,y) &=\E[u(x)\partial_\nu u(y)],\\
F_3(x,y)&=\E[\partial_\nu u(x) \partial_\nu u(y)] ,
\end{align*}
where $x,y\in\partial B_R$. Let 
\begin{equation*}%\label{eq:defM}
 M = \max\left\{ \Vert F_1\Vert_{L^2(\partial B_R\times \partial B_R)},\, \Vert F_2\Vert_{L^2(\partial B_R\times \partial B_R)},\, \Vert F_3\Vert_{L^2(\partial B_R\times \partial B_R)}\right\}. 
 \end{equation*}
 
 There are two types of problems to be studied. Given the strength function $\mu$ of the random source, the direct source problem involves examining the well-posedness of either \eqref{eqn:mainh} and \eqref{eqn:rc} in a homogeneous medium or \eqref{eqn:main} and \eqref{eqn:rc} in an inhomogeneous medium. The inverse source problem focuses on determining the strength $\mu$ of the random source $f$ based on the Cauchy data $u$ and $\partial_\nu u$ of the wave field measured on  the boundary $\partial B_R$. The goal of this work is to establish stability estimates for the inverse source problem of the stochastic Helmholtz equation in both homogeneous and inhomogeneous media. 
 
Our main results are summarized in the following two theorems. The former is concerned with the stability of the inverse source problem in a homogeneous medium, while the latter concerns the stability of the inverse source problem in an inhomogeneous medium.

\begin{theorem}\label{thm:white_homo}
 There exists a positive constant $k_0$ depending on $s, M, R$ such that the following estimate holds for $k>k_0$:
\begin{equation}\label{est:homowhite}
\Vert \mu\Vert_{L^\infty(B_1)}\leq  C_1 M^{1-\frac{3}{s}},
\end{equation}
where $C_1$ is a positive constant depending on  $s, k, R$ and satisfies 
\[
 C_1=O(k^{5(1-\frac{3}{s})}),\quad k\to\infty. 
\]
\end{theorem}

\begin{theorem}\label{thm:white_inhomo}
There exists a positive constant $C_2$ depending on  $s, k, R$ such that the following estimate holds for $k>0$:
\begin{equation}\label{est:inhomowhite}
\Vert \mu\Vert_{L^\infty(B_1)}\leq C_2\left(\ln(3+M^{-1})\right)^{1-\frac{s}{3}},
\end{equation}
where $C_2$ satisfies
\[
 C_2=O(k^{5(\frac{s}{3}-1)}),\quad k\to\infty. 
\]

\end{theorem}

For the first time, we show the stability of the inverse random source problems by using the correlation of the Cauchy data at a fixed frequency. By Theorems \ref{thm:white_homo} and \ref{thm:white_inhomo}, the following observations can be made. First, the stability estimates \eqref{est:homowhite} and \eqref{est:inhomowhite} imply the uniqueness of the inverse random source problems for the Helmholtz equation at a fixed frequency. Previous uniqueness results typically require multi-frequency data for the inverse source problems in deterministic settings. The difference in our findings may be attributed to the characteristics of the random source and the data itself. Specifically, the white noise acts as a radiating source, while the type of correlation data can offer additional insights into the unknowns. Second, the inverse random source problems have a H\"{o}lder stability and a logarithmic stability for homogeneous and inhomogeneous media, respectively. Moreover, the H\"{o}lder stability may approach to the optimal Lipschitz stability if the strength function $\mu$ is sufficiently smooth. Third, the estimates \eqref{est:homowhite} and \eqref{est:inhomowhite} provide the explicit dependence of the constants on the wavenumber $k$, which is of practical significance but received limited attention in the literature on the stability estimates of the inverse medium problems \cite{hahner2001new,isaev2013new}.

The structure of the paper is as follows. In Section \ref{sec:2}, we address the well-posedness of the direct
source problems. Section \ref{sec:3} is devoted to inverse source problems where the stability estimates are established for both cases of the homogeneous and inhomogeneous media. The paper concludes by providing general observations and remarks in Section \ref{sec:4}.

\section{The direct problems}\label{sec:2}

In this section, we focus on examining the well-posedness of the direct problems, specifically the validation of the Cauchy data $u$ and $\partial_\nu u$ on the boundary $\partial B_R$.  

Throughout, we use $C$ to denote a generic constant and $c_j$ to represent a constant in some intermediate step. The value of the generic constant is not required and may change step by step, but its meaning will always be clear from the context.

\begin{lemma}
Let $s > 3$ and consider the function $\mu \in H_0^s(B_1)$ with $\mu\geq 0$. Then, it follows that $\sqrt{\mu} \in C^{0,\eta}(B_1)$ for some $\eta \in (0, \frac{1}{2})$.
\end{lemma}

\begin{proof}
Applying the Sobolev embedding theorem, we know that $W^{s_0,r}(\mathbb{R}^3)\subset C^{0,\alpha_0}(\mathbb{R}^3)$ holds when $3 < s_0r$ and $\alpha_0 = s_0 - \frac{3}{r} \in (0,1)$. By choosing $r=2$, we obtain $H^{s_0}(\mathbb{R}^3)\subset C^{0, s_0-\frac{3}{2}}(\mathbb{R}^3)$ for $s_0\in \left(\frac{3}{2}, \frac{5}{2}\right)$. Since $s> 3$, we can deduce that $H^s(\mathbb{R}^3)\subset C^{0,1-\varepsilon}(\mathbb{R}^3)$ for any $\varepsilon\in (0,1)$. Consequently, for $\mu\in H_0^s(B_1)$, we have $\mu\in C^{0,1-\varepsilon}(\mathbb{R}^3)$. Moreover, since $\mu$ has compact support within $B_1$, it follows that $\mu\in C^{0,1-\varepsilon}(B_1)$. Thus, there exists a positive constant $C$ such that 
for all $x,y \in B_1$, 
\[
 |\mu(x)-\mu(y)|\leq C|x-y|^{1-\varepsilon},
\]
which gives for $\mu(x)\neq 0$ or $\mu(y)\neq 0$ that 
\[
 |\sqrt{\mu}(x)-\sqrt{\mu}(y)|=\frac{|\mu(x)-\mu(y)|}{|\sqrt{\mu}(x)+\sqrt{\mu}(y)|}\leq \frac{C|x-y|^{1-\varepsilon}}{|\sqrt{\mu}(x)+\sqrt{\mu}(y)|}. 
\]
It is clear to note 
\begin{align*}
|\sqrt{\mu}(x)-\sqrt{\mu}(y)|\leq  \min\left\{ |\sqrt{\mu}(x)+\sqrt{\mu}(y)|, \, \frac{C|x-y|^{1-\varepsilon}}{|\sqrt{\mu}(x)+\sqrt{\mu}(y)|}\right\}\leq \sqrt{C}|x-y|^{\frac{1}{2}-\frac{\varepsilon}{2}}, 
\end{align*}
which shows that $\sqrt{\mu}$ belongs to the H\"{o}lder space $C^{0,\eta}(B_1)$ with $\eta\in(0,\frac{1}{2})$.
\end{proof}

It is shown in \cite[Theorem 2.7]{BCL16} that the direct problem \eqref{eqn:mainh} and \eqref{eqn:rc} is well-posed. The obtained well-posedness result is summarized in the following lemma. 

\begin{lemma}%\label{lem:uniqueh}
 For any $k>0$, the direct problem \eqref{eqn:mainh} and \eqref{eqn:rc} has a unique, continuous stochastic process $u: B_R\to \mathbb{C}$ that fulfills
\begin{equation}\label{eq:repre_homo}
u(x) = \mathcal{G}[\sqrt{\mu}\dot{W}](x),\quad   \mathbb{P}\text{-}a.s.,
\end{equation}
where
\[
\mathcal{G}[\sqrt{\mu} \dot{W}](x) = \int_{B_1} G(x,y) \sqrt{\mu}(y) dW_y
\]
with $G$ being the Green's function of the three-dimensional Helmholtz equation, i.e., 
\[
G(x,y) = -\frac{1}{4\pi} \frac{e^{\ii k|x-y|}}{|x-y|}.
\]
\end{lemma}

In \cite[Theorem 2.4]{li2017inverse}, it is demonstrated that for inhomogeneous media, the direct problem \eqref{eqn:main} and \eqref{eqn:rc} admits a unique mild solution in two dimensions. By adapting the proof, the well-posedness of the direct problem \eqref{eqn:main} and \eqref{eqn:rc} can be extended to three dimensions. The well-posedness is summarized in the subsequent lemma, and for the sake of completeness, a brief sketch of the proof is provided.

\begin{lemma}\label{lem:unique}
For any $k>0$, the direct problem \eqref{eqn:main} and \eqref{eqn:rc} has a unique, continuous stochastic process $u: B_R\to \mathbb{C}$, which satisfies
\begin{equation}\label{eq:repre_inhomo}
u(x) = -k^2 \mathcal{G}[qu](x) + \mathcal{G}[\sqrt{\mu}\dot{W}](x),\quad  \mathbb{P}\text{-}a.s.,
\end{equation}
where
\[
\mathcal{G}[qu](x) = \int_{B_1} G(x,y) q(y)u(y) dy,\quad  \mathcal{G}[\sqrt{\mu} \dot{W}](x) = \int_{B_1} G(x,y) \sqrt{\mu}(y) dW_y. 
\]
\end{lemma}

\begin{proof}
First we employ a piecewise constant approximation for the white noise $\dot{W}_x$ (cf. \cite{cao2008finite}). Let $\mathcal{T}_N$ be the union of tetrahedra $K_j$, which form a regular triangulation of $B_1$. Denote a piecewise function
\begin{equation}\label{eq:defW}
\dot{W}_x^N = \sum_{j=1}^N |K_j|^{-1} \int_{K_j} dW_x \chi_j(x),
\end{equation}
where $|K_j|$ is the volume of the tetrahedron $K_j$ and $\chi_j$ represents the characteristic function of $K_j$. We define a sequence $\{ u^N_0\}$ as follows: 
\begin{equation*}
u^N_0(x):= \mathcal{G}[\sqrt{\mu}\dot{W}^N](x) =\int_{B_1} G(x,y) \sqrt{\mu}(y)dW^N_y. 
\end{equation*}
It follows from the proof of \cite[Theorem 2.7]{BCL16} that  
\[
\E \left[ \int_{B_1} |u_0(x)-u_0^N(x)| ^2 dx\right]\to 0,
\]
where
\begin{equation*}
u_0(x):= \mathcal{G}[\sqrt{\mu}\dot{W}](x) =\int_{B_1} G(x,y) \sqrt{\mu}(y)dW_y
\end{equation*}
and there exists a continuous modification of the random field $u_0$.

To establish the existence of a solution, we show that the solution $u^N$ to 
\begin{equation}\label{eqn:uN}
\Delta u^N+ k^2(1+q)u^N =  \sqrt{\mu} \dot W^N
\end{equation}
converges to the random field $u$ in \eqref{eq:repre_inhomo}. Since $ \sqrt{\mu} \dot W^N\in L^2(B_R)$, it follows from     \cite[Theorem 2.3]{bao2021adaptive} that there exists a unique weak solution $u^N\in H^1(B_R)$. By \eqref{eqn:uN}, it is clear to note that $u^N$ satisfies the Lippmann--Schwinger integral equation
\begin{equation*}
u^N(x) + k^2 \mathcal{G}[qu^N](x) = \mathcal{G}[\sqrt{\mu}\dot{W}^N](x).
\end{equation*}
Hence we have for any constants $N_1$ and $N_2$ that  
  \begin{equation}\label{eq:diff_approx}
(u^{N_1} - u^{N_2}) +  k^2 \mathcal{G}[q(u^{N_1}-u^{N_2})](x) = \mathcal{G}[\sqrt{\mu}(\dot{W}^{N_1}-\dot{W}^{N_2})](x).
\end{equation}
Following the proof of \cite[Theorem 2.3]{bao2021adaptive}, we may obtain from \eqref{eq:diff_approx} that 
\begin{align*}
\Vert u^{N_1} - u^{N_2}\Vert_{L^2(B_R)} &\leq C \Vert  \mathcal{G}[\sqrt{\mu}(\dot{W}^{N_1}-\dot{W}^{N_2})]\Vert_{L^2(B_R)}\\
&\leq C  \Vert u_0^{N_1}-u_0^{N_2}\Vert_{L^2(B_R)},
\end{align*}
which implies that the sequence $\{u^N\}$ is convergent due to the convergence of $\{u_0^N\}$. Denoting the limit of $u^N$ by $u$, we prove the existence of a continuous mild solution.

To prove the uniqueness of the solution to \eqref{eqn:main} and \eqref{eqn:rc}, let $u_1$ and $u_2$ be two solutions to \eqref{eqn:main} and \eqref{eqn:rc}. Define $\tilde{u}=u_1-u_2$, then $\tilde{u}$ satisfies the homogeneous Helmholtz equation
\[
\Delta \tilde{u} + k^2(1+q)\tilde{u}=0
\]
and the Sommerfeld radiation condition. The proof is completed by utilizing the uniqueness result established in \cite[Theorem 2.3]{bao2021adaptive}, which implies that $\tilde{u}=0$. 
\end{proof}

By Lemma \ref{lem:unique}, the solution to the direct problem \eqref{eqn:main} and \eqref{eqn:rc} is continuous in $B_R$. The following result demonstrates that the Cauchy data $u$ and $\partial_\nu u$ are well-defined on $\partial B_R$. 

\begin{lemma}\label{lem:boundary}
The solution $u$ to \eqref{eqn:main} and \eqref{eqn:rc} is analytic in $\mathbb{R}^3\setminus \overline{B_1}$. 
\end{lemma}

\begin{proof}
It follows from Lemma \ref{lem:unique} that the direct problem \eqref{eqn:main} and \eqref{eqn:rc} admits a unique continuous mild solution $u$ in $B_R$.  Consider the following exterior Dirichlet problem:
\begin{equation*} 
\begin{split}
\Delta \bar{u} +k^2\bar{u} = 0\quad  &\text{in } \mathbb{R}^3\setminus\overline{B_1},\\
\bar{u} = u|_{\partial B_1} \quad &\text{on } \partial B_1,
\end{split}
\end{equation*}
where $\bar{u}$ satisfies the Sommerfeld radiation condition. The solution $\bar u$ admits the Fourier series expansion
\[
 \bar u(x)=\sum_{n=0}^\infty\sum_{m=-n}^m \frac{h_n^{(1)}(kr)}{h_n^{(1)}(k)}\bar u_n^m(\hat x)Y_n^m(\hat x),\quad r=|x|>1,
\]
where $\hat x=x/|x|, h_n^{(1)}$ is the spherical Hankel function of the first kind with order $n$, $Y_n^m$ is a spherical harmonic function of degree $n$ and order $m$, and the Fourier coefficient 
\[
 \bar u_n^m(\hat x)=\int_{|x|=1} u(\hat x)\overline{Y_n^m(\hat x)}ds. 
\]
The proof is completed by noting that $u$ is continuous on $\partial B_1$ and the exterior Dirichlet problem has a unique solution in $\mathbb R^3\setminus\overline{B_1}$. 
\end{proof}

Clearly, it follows from Lemma \ref{lem:boundary} that the solution to \eqref{eqn:mainh} and \eqref{eqn:rc} is also analytic in  $\mathbb{R}^3\setminus \overline{B_1}$. Therefore, the Cauchy data  $u$ and $\partial_\nu u$ are well-defined on $\partial B_R$, which validate that $F_j$, $j=1,2,3$ are well-defined on $\partial B_R\times \partial B_R$.

\section{The inverse problems}\label{sec:3}

In this section, we establish the stability estimates by estimating the Fourier coefficients of the strength $\mu$ separately in low and high frequency modes. The idea traces back to the stability estimates for the inverse medium scattering problem in a deterministic setting \cite{hahner2001new,isaev2013new}. 

Denote by $\hat{\mu}$ the Fourier transform of $\mu$, i.e., 
\begin{equation*}
\hat{\mu}(\gamma) = (2\pi)^{-3}\int_{\mathbb{R}^3} e^{-\ii\gamma\cdot x}\mu(x)dx,\quad\gamma\in\mathbb R^3.
\end{equation*}
The following result, known as the Paley--Wiener--Schwartz theorem, provides decay properties of the Fourier transform for a function with compact support.

\begin{lemma}\label{PWS}
For $\mu\in H^{s}_0(B_1)$,  there exists a constant $c_0$ depending only on $s$ such that 
\begin{equation*}
|\hat{\mu}(\gamma)|\leq c_0 (1+|\gamma|)^{-s}.
\end{equation*}
\end{lemma}

For a positive constant $\rho$, it is clear to note from Lemma \ref{PWS} that there exists a constant $c_1$ depending on $s$ such that the following estimate holds: 
\begin{equation}\label{eqn:fourierhigh}
\int_{|\gamma|>\rho} |\hat{\mu}(\gamma)|d\gamma\leq c_1\frac{1}{\rho^{s-3}}.
\end{equation}

Next, our investigation focuses on estimating the Fourier coefficients $\hat{\mu}(\gamma)$ with $|\gamma|\leq \rho$ for the inverse problems in homogeneous and inhomogeneous media, respectively. To achieve this, we introduce special solutions to the Helmholtz equation, which enables us to establish the relationship between the measurement $M$ and the Fourier coefficients of $\mu$.

\subsection{Homogeneous media}

First, we address the stability of the inverse random source problem for the Helmholtz equation within a homogeneous medium, wherein the plane wave serves as the chosen special solution.

The following lemma plays a crucial role in the stability estimation as it establishes a link between the internal distribution of the random source and the measurement of the wave field $u$ on $\partial B_R$.

\begin{lemma}\label{lem:Ewhitehomo}
For two solutions $U_i\in C^3\left(\overline{B_{\hat R}}\right)$, $i=1,2$, of $\Delta u+ k^2 u = 0$ with $\hat R>R>1$, there exists a positive constant $c_2$ that depends on $k$, $R$, and $\hat R$, satisfying
\begin{equation}\label{Ewhitehomo-s1}
|\E[\langle f, U_1\rangle \langle f, U_2\rangle]| \leq c_2 M \Vert U_1\Vert_{L^2(B_{\hat R})} \Vert U_2\Vert_{L^2(B_{\hat R})}. 
\end{equation}
\end{lemma}

\begin{proof}
The proof can be divided into two parts. First, we establish the estimate \eqref{Ewhitehomo-s1} within a finite-dimensional setting. 

Let $f^N = \sqrt{\mu}\dot{W}_x^N$, where $\dot{W}_x^N$ is defined in \eqref{eq:defW}. It follows from the proof in  \cite[Theorem 2.7]{BCL16} that $\dot{W}_x^N\in L^p(B_1)$ for any $p>1$. We consider the Helmholtz equation $\Delta u^N +k^2 u^N = f^N$ together with the Sommerfeld radiation condition for $u^N$. Employing the equations $\Delta U_i+k^2 U_i=0$ for $i=1, 2$ and applying the integration by parts yields 
\begin{align*}
 \langle f^N, U_i\rangle  
 &= \int_{B_R} f^N(x) U_i(x)dx\\
 &= \int_{B_R}\left( \Delta u^N(x) +k^2 u^N(x)\right) U_i(x)dx\\
 &= \int_{\partial B_R} \left( \partial_\nu u^N(x) U_1(x) -\partial_\nu U_1(x) u^N(x)\right) ds(x).
 \end{align*}
Using the above equation, we obtain from straightforward calculations that 
\begin{align}
&\mathbb{E}\left[\langle f^N, U_1\rangle \langle f^N, U_2\rangle\right]  \nonumber\\
%&= \int_{\partial B_R} \int_{\partial B_R}\mathbb{E}\left[\left( \partial_\nu u^N(x) U_1(x) -\partial_\nu U_1(x) u^N(x)\right) \left( \partial_\nu u^N(y) U_2(y) -\partial_\nu U_2(y) u^N(y)\right)  \right]ds(x)ds(y)\nonumber\\
&= \int_{\partial B_R} \int_{\partial B_R}\mathbb{E}[\left( \partial_\nu u^N(x) U_1(x) -\partial_\nu U_1(x) u^N(x)\right) \nonumber\\
&\qquad \qquad\times \left( \partial_\nu u^N(y) U_2(y) -\partial_\nu U_2(y) u^N(y)\right)  ]ds(x)ds(y)\nonumber\\
&\leq\int_{\partial B_R} \int_{\partial B_R} \left|\mathbb{E}[\partial_\nu u^N(x) \partial_\nu u^N(y)] U_1(x)  U_2(y)\right| ds(x) ds(y)\nonumber\\
&\quad +\int_{\partial B_R} \int_{\partial B_R}\left|\mathbb{E}[\partial_\nu u^N(x)  u^N(y) ]U_1(x)  \partial_\nu U_2(y) \right|ds(x) ds(y)\nonumber\\
&\quad +\int_{\partial B_R} \int_{\partial B_R}\left|\mathbb{E}[ u^N(x) \partial_\nu u^N(y) ]\partial_\nu U_1(x)  U_2(y)\right|ds(x) ds(y)\nonumber\\
&\quad +\int_{\partial B_R} \int_{\partial B_R} \left|\mathbb{E}[ u^N(x)  u^N(y)] \partial_\nu U_1(x)  \partial_\nu U_2(y)\right| ds(x) ds(y)\nonumber\\
&\leq  M^N\Vert U_1\Vert_{L^2(\partial B_R)}\Vert U_2\Vert_{L^2(\partial B_R)}+M^N\Vert U_1\Vert_{L^2(\partial B_R)}\Vert \partial_\nu U_2\Vert_{L^2(\partial B_R)}\nonumber\\
&\quad +M^N\Vert \partial_\nu U_1\Vert_{L^2(\partial B_R)}\Vert  U_2\Vert_{L^2(\partial B_R)}+M^N\Vert \partial_\nu U_1\Vert_{L^2(\partial B_R)}\Vert \partial_\nu U_2\Vert_{L^2(\partial B_R)},\label{eq:est_E}
\end{align}
where 
$$ M^N = \max\{ \Vert F^N_1\Vert_{L^2(\partial B_R\times \partial B_R)},\, \Vert F^N_2\Vert_{L^2(\partial B_R\times \partial B_R)},\, \Vert F^N_3\Vert_{L^2(\partial B_R\times \partial B_R)}\}$$
and
\begin{align*}
F^N_1(x,y)&=\E[u^N(x) u^N(y)],\\
F^N_2(x,y) &=\E[u^N(x)\partial_\nu u^N(y)],\\
F^N_3(x,y)&=\E[\partial_\nu u^N(x) \partial_\nu u^N(y)] .
\end{align*}

As $U_i$ solves $\Delta u +k^2 u =0$ in $B_{\hat R}$, and given that the wavenumber $k$ is a positive constant, the real and imaginary parts of $U_i$, denoted as $\Re U_i$ and $\Im U_i$, respectively, also belong to $C^3(\overline{B_{\hat R}})$ and satisfy $\Delta u +k^2 u =0$. By assuming that $|\Re U_i|$ and $|\Im U_i|$ attain their maximum values on $\partial B_R$ at points $x_1$ and $x_2$ respectively, we apply the mean value property of solutions to the Helmholtz equation to derive the following identities (cf. \cite{kuznetsov2021mean}):
\begin{align}
\Re U_i(x_1) &= h(kr)  \frac{1}{| B_{x_1,r}|}\int_{B_{x_1,r}} \Re U_i(x) dx, \label{eq:reU}\\
\Im U_i(x_2) &=  h(kr) \frac{1}{| B_{x_2,r}|} \int_{B_{x_2,r}}\Im U_i(x) dx.\label{eq:imU}
\end{align}
Here, $h(kr)$ is defined as 
\[
h(kr) =\frac{4}{3\sqrt{\pi}}\frac{(\frac{kr}{2})^{\frac{3}{2}}}{J_\frac{3}{2} ( kr)}, 
\]
where $J_\frac{3}{2}$ represents Bessel's function given by
\begin{equation*}
J_\frac{3}{2}(t) = \sqrt{\frac{2}{\pi t}}\left(\frac{\sin t}{t}-\cos t\right),
\end{equation*}
and  $B_{x_j,r}\subset B_{\hat R-\epsilon}\setminus\overline{B_1}$, $j=1, 2,$ denotes a ball centered at  $x_j$ with radius $r$, where $\epsilon<\frac{\hat R-R}{2}$. Since $h(1) \leq 2$, taking the radius of the ball as $r = \min\{\frac{1}{k}, \frac{\hat R-R}{2}, \frac{R-1}{2} \}$ in \eqref{eq:reU}--\eqref{eq:imU} and using the Cauchy--Schwartz inequality leads to  
\begin{align*}
\Vert \Re U_i\Vert_{L^\infty(\partial B_R)} &\leq   \sqrt{\frac{3}{\pi} }   r^{-\frac{3}{2}} \Vert \Re U_i(x)\Vert_{L^2(B_{\hat R-\epsilon}\setminus\overline{B_1})},\\
\Vert \Im U_i\Vert_{L^\infty(\partial B_R)} &\leq  \sqrt{\frac{3}{\pi} }   r^{-\frac{3}{2}} \Vert \Im U_i(x)\Vert_{L^2(B_{\hat R-\epsilon}\setminus\overline{B_1})}. 
\end{align*}
A simple calculation gives
\begin{align*}
\Vert U_i\Vert_{L^2(\partial B_R)}&\leq \sqrt{4\pi R^2 \left(\Vert \Re U_i\Vert_{L^\infty(\partial B_R)}^2 +\Vert \Im U_i\Vert_{L^\infty(\partial B_R)}^2\right)} \\
&\leq  2  \sqrt{3 }    Rr^{-\frac{3}{2}}  \Vert U_i(x)\Vert_{L^2(B_{\hat R-\epsilon}\setminus\overline{B_1})}.
\end{align*}

Since the first-order derivatives of $U_i$ also satisfy the Helmholtz equation and belong to $C^2(\overline{B_{\hat R}})$, we 
may follow the same arguments and show that 
\begin{eqnarray*}
\Vert \partial_\nu U_i\Vert_{L^2(\partial B_R)}\leq 2  \sqrt{3}R r^{-\frac{3}{2}} \Vert \nabla U_i(x)\Vert_{L^2(B_{\hat R-\epsilon}\setminus\overline{B_1})}.
\end{eqnarray*}
By the Caccioppoli's inequality for the Helmholtz equation \cite[Lemma 3.3]{berge2021three}, there exists a positive constant $c_3$ such that 
\begin{equation*}
\Vert \nabla U_i(x)\Vert^2_{L^2(B_{\hat R-\epsilon}\setminus\overline{B_1})}\leq\left(k^2 +\frac{c_3}{\epsilon^2}\right) \Vert U_i(x)\Vert^2_{L^2(B_{\hat R})},
\end{equation*}
which gives 
\begin{align}
\Vert  U_i\Vert_{L^2(\partial B_R)}&\leq 2\sqrt{3 } Rr^{-\frac{3}{2}} \Vert U_i(x)\Vert_{L^2(B_{\hat R})}, \label{est:U1} \\
 \Vert \partial_\nu U_i\Vert_{L^2(\partial B_R)}&\leq 2  \sqrt{3 } Rr^{-\frac{3}{2}}  \left(k^2 +\frac{c_3}{\epsilon^2}\right)^{1/2}\Vert U_i(x)\Vert_{L^2(B_{\hat R})}.\label{est:U2}
\end{align}
Substituting \eqref{est:U1}--\eqref{est:U2} into \eqref{eq:est_E}, we arrive at
\begin{equation}\label{eq:estEN}
\E[\langle f^N, U_1\rangle \langle f^N, U_2\rangle]\leq c_2 M^N \Vert U_1\Vert_{L^2(B_{\hat R})} \Vert U_2\Vert_{L^2(B_{\hat R})},
\end{equation}
where the positive constant $c_2 = 12 R^2 r^{-3} \left(1+ \left(k^2 +\frac{c_3}{\epsilon^2}\right)^{1/2}\right)^2. $

Next, we proceed with the proof of the estimate \eqref{Ewhitehomo-s1} by introducing a sequence of approximations denoted as ${f^N}$, where $f^N= \sqrt{\mu}\dot{W}_x^N$. We construct a triangulation $\mathcal{T}_N$, which satisfies $\max_{K_j\in \mathcal{T}_N}\text{diam}(K_j)\to 0$ as $N\to \infty$. The goal is to show that $\E[\langle f^N, U_1\rangle \langle f^N, U_2\rangle]\to \E[\langle f, U_1\rangle \langle f, U_2\rangle]$ and $M^N\to M$ as $N\to\infty$.

By utilizing the definition of $f^N$ and applying the It\^{o} isometry, we obtain
\begin{align}\label{Ewhitehomo-s2}
&\E[\langle f^N, U_1\rangle \langle f^N, U_2\rangle]\nonumber\\
&=\int_{B_1}\int_{B_1} U_1(x)U_2(y) \sqrt{\mu}(x)\sqrt{\mu}(y)\E[\dot{W}_x^N \dot{W}_y^N ]dxdy\nonumber\\
&=\int_{B_1}\int_{B_1} U_1(x)U_2(y) \sqrt{\mu}(x)\sqrt{\mu}(y) \nonumber\\
&\qquad\qquad\times\E\left[ \sum_{j=1}^N |K_j|^{-1} \int_{K_j} dW_x \chi_j(x)  \sum_{l=1}^N |K_l|^{-1} \int_{K_l} dW_y \chi_l(y)\right]dxdy\nonumber\\
&=\int_{B_1}\int_{B_1} U_1(x)U_2(y) \sqrt{\mu}(x)\sqrt{\mu}(y) \nonumber\\
&\qquad\qquad\times\sum_{j=1}^N  \sum_{l=1}^N  |K_j|^{-1}|K_l|^{-1}\chi_j(x)\chi_l(y) \E\left[ \int_{K_j}dW_x   \int_{K_l} dW_y \right]dxdy\nonumber\\
&=\int_{B_1}\int_{B_1} U_1(x)U_2(y) \sqrt{\mu}(x)\sqrt{\mu}(y)\sum_{j=1}^N   |K_j|^{-2}\chi_j(x)\chi_j(y)  |K_j| dxdy\nonumber\\
&= \sum_{j=1}^N \int_{K_j}\int_{K_j} U_1(x)U_2(y) \sqrt{\mu}(x)\sqrt{\mu}(y)   |K_j|^{-1}   dxdy. 
\end{align}
On the other hand, an application of the It\^{o} isometry gives
\begin{align}\label{Ewhitehomo-s3}
\E[\langle f, U_1\rangle \langle f, U_2\rangle]&=\E\left[\langle \dot{W}_x, \sqrt{\mu}U_1\rangle \langle \dot{W}_y, \sqrt{\mu}U_2\rangle\right]\nonumber\\
&=\int_{B_1} \mu(x)  U_1(x)U_2(x)dx.
\end{align}
Combining \eqref{Ewhitehomo-s2} and \eqref{Ewhitehomo-s3} yields 
\begin{align}
&\E[\langle f, U_1\rangle \langle f, U_2\rangle]-\E[\langle f^N, U_1\rangle \langle f^N, U_2\rangle] \nonumber\\
&=\int_{B_1} U_1(x)U_2(x) \mu(x) dx-\sum_{j=1}^N \int_{K_j}\int_{K_j} U_1(x)U_2(y) \sqrt{\mu}(x)\sqrt{\mu}(y)   |K_j|^{-1}   dxdy  \nonumber\\
&= \sum_{j=1}^N |K_j|^{-1}   \int_{K_j} \int_{K_j}   U_1(x)U_2(x) \mu(x) dxdy \nonumber\\
&\quad - \sum_{j=1}^N   |K_j|^{-1} \int_{K_j}\int_{K_j} U_1(x)U_2(y) \sqrt{\mu}(x)\sqrt{\mu}(y)  dxdy  \nonumber\\
&= \sum_{j=1}^N |K_j|^{-1}   \int_{K_j} \int_{K_j}   U_1(x)\sqrt{\mu}(x)( U_2(x) \sqrt{\mu}(x)-U_2(y) \sqrt{\mu}(y) ) dxdy  \nonumber\\
&\leq\sum_{j=1}^N |K_j|^{-1}   \int_{K_j} \int_{K_j}  \left| U_1(x)\sqrt{\mu}(x)\right| \left|U_2(x) \sqrt{\mu}(x)-U_2(y) \sqrt{\mu}(y) \right|dxdy \label{eq:estEE1}.
\end{align}
Since $U_2\in C^3(\overline{B_{\hat R}})$ and $\sqrt{\mu}\in C^{0,\eta}(B_1)$, there exists a constant $C$ depending on $U_2$ such that
\[
|U_2(x) \sqrt{\mu}(x)-U_2(y) \sqrt{\mu}(y) | \leq C |x-y|^\eta\leq C |\text{diam}(K_j)|^\eta\quad\forall\, x, y\in K_j.
\]
 Substituting the above inequality into \eqref{eq:estEE1} gives 
\begin{align*}
&\E[\langle f, U_1\rangle \langle f, U_2\rangle]-\E[\langle f^N, U_1\rangle \langle f^N, U_2\rangle]\\
&\leq C\sum_{j=1}^N |K_j|^{-1}   |\text{diam}(K_j)|^\eta \int_{K_j} \int_{K_j}  | U_1(x)\sqrt{\mu}(x)| dxdy\\
&\leq C \Big|\max_{K_j\in\mathcal{T}_N}\text{diam}(K_j)\Big|^\eta\sum_{j=1}^N  \int_{K_j}  | U_1(x)\sqrt{\mu}(x)| dx\\
&\leq C \Big|\max_{K_j\in\mathcal{T}_N}\text{diam}(K_j)\Big|^\eta \Vert  U_1\Vert_{L^2(B_{\hat R})}  \Vert  \sqrt{\mu}\Vert_{L^2(B_1)}.
\end{align*}
Given that $\Vert U_1\Vert_{L^2(B_{\hat R})}$ and $ \Vert \sqrt{\mu}\Vert_{L^2(B_1)}$ are bounded, and observing that as $N\to \infty$,  $|\max_{K_j\in\mathcal{T}_N}\text{diam}(K_j)|^\eta\to 0$, we can deduce that  
\begin{equation}\label{eq:convE}
\E[\langle f^N, U_1\rangle \langle f^N, U_2\rangle]\to \E[\langle f, U_1\rangle \langle f, U_2\rangle]\quad\text{as } N\to\infty. 
\end{equation}

To establish the convergence $M_N\to M$ as $N\to\infty$, it is sufficient to verify the convergence $\Vert F^N_j\Vert_{L^2(\partial B_R\times \partial B_R)}\to \Vert F_j\Vert_{L^2(\partial B_R\times \partial B_R)}$ for $j=1, 2, 3$ as $N\to\infty$. It is evident from the expression of the solution in \eqref{eq:repre_homo} and the Itô isometry that
\begin{align*}
\E[u^N(x)u^N(y)]&= \E\left[\int_{B_1} G(x,z_1)\sqrt{\mu}(z_1)dW_{z_1}^N  \int_{B_1} G(y,z_2)\sqrt{\mu}(z_2)dW_{z_2}^N\right]\\
&=\sum_{j=1}^N   |K_j|^{-1}\int_{K_j}\int_{K_j} G(x,z_1) G(y,z_2) \sqrt{\mu}(z_1) \sqrt{\mu}(z_2)dz_1 dz_2
\end{align*}
and
\begin{align*} 
\E[u(x)u(y)] &=  \E\left[\int_{B_1} G(x,z_1)\sqrt{\mu}(z_1)dW_{z_1}  \int_{B_1} G(y,z_2) \sqrt{\mu}(z_2)dW_{z_2}\right]\\
&=\int_{B_1} G(x,z_1) G(y,z_1)\mu(z_1)dz_1,
\end{align*}
which lead to 
\begin{align*}
&\E[u(x)u(y)]-\E[u^N(x)u^N(y)] \\
&= \sum_{j=1}^N  |K_j|^{-1}\int_{K_j}\int_{K_j} G(x,z_1)  \sqrt{\mu}(z_1) (G(y,z_1)\sqrt{\mu}(z_1)- G(y,z_2)\sqrt{\mu}(z_2)) dz_1 dz_2.
\end{align*}
Since $G(y,\cdot)\in C^2(B_1)$ for $y\in \partial B_R$ and $\sqrt{\mu}\in C^{0,\eta}(B_1)$, there exists a constant $C$ depending on $k$ and $R$ such that 
\[
\left|G(y,z_1)\sqrt{\mu}(z_1)- G(y,z_2)\sqrt{\mu}(z_2)\right| \leq  C|z_1-z_2|^\eta\leq C|\text{diam}(K_j)|^\eta\quad\forall\, z_1,z_2\in K_j. 
\]
Clearly, there exists a uniform upper bound on $\Vert G(x,\cdot) \Vert_{L^2(B_1)}$ for all $x\in \partial B_R$, which depends on $R$. Consequently, a simple calculation yields that
\begin{align*}
\E[u(x)u(y)]-\E[u^N(x)u^N(y)] &\leq \sum_{j=1}^N C |\text{diam}(K_j)|^\eta  |K_j|^{-1}   \nonumber\\
&\qquad\qquad\times\int_{K_j}\int_{K_j}| G(x,z_1)  \sqrt{\mu}(z_1) |dz_1 dz_2\\
&\leq C \Big|\max_{K_j\in\mathcal{T}_N}\text{diam}(K_j)\Big|^\eta \int_{B_1} | G(x,z_1)  \sqrt{\mu}(z_1) |dz_1\\
&\leq C  \Big|\max_{K_j\in\mathcal{T}_N}\text{diam}(K_j)\Big|^\eta \Vert G(x,\cdot) \Vert_{L^2(B_1)}  \Vert \sqrt{\mu}\Vert_{L^2(B_1)}\\
&\leq C  \Big|\max_{K_j\in\mathcal{T}_N}\text{diam}(K_j)\Big|^\eta  \Vert \sqrt{\mu}\Vert_{L^2(B_1)},
\end{align*}
 which leads to 
\begin{eqnarray*}
\Vert F_1 -F^N_1\Vert_{L^2(\partial B_R\times \partial B_R)}\leq 4\pi R^2 C\Vert \sqrt{\mu} \Vert_{L^2(B_1)}   |\max_{K_j\in\mathcal{T}^N}\text{diam}(K_j)|^\eta\to 0\quad \text{as } N\to\infty.  
\end{eqnarray*}

Similarly, utilizing the H\"{o}lder continuity of $\partial_\nu G(y,\cdot) \sqrt{\mu}(\cdot)$ and the boundedness of $\Vert \nabla G(x,\cdot) \Vert_{L^2(B_1)}$, we can demonstrate the convergence of $F_2^N$ and $F_3^N$. As a result, $M^N\to M$. By combining the convergence of $M^N$ with \eqref{eq:estEN} and \eqref{eq:convE}, we arrive at the desired conclusion.
\end{proof}

Next, we employ special solutions to the Helmholtz equation in a homogeneous medium to provide an estimation of the Fourier coefficients $\hat{\mu}(\gamma)$ for $|\gamma|\leq 2k$.

\begin{lemma}\label{lem:Fourierlow_white_homo}
There exists a positive constant $c_4$  depending on $k$, $R$, and $\hat R$ such that the following estimate holds
for all $\gamma\in\mathbb{R}^3$ with $|\gamma|< 2k$:
\[
|\hat{\mu}(\gamma)|\leq c_4 M.
\]
\end{lemma}

\begin{proof}
First we construct two special solutions to the Helmholtz equation $\Delta u +k^2 u =0$.  Given a fixed $\gamma\in \mathbb{R}^3$, we select a unit vector $d\in\mathbb{R}^3$ such that $d\cdot \gamma = 0$, and then define two real vectors
\begin{align*}
\xi^{(1)} &= -\frac{1}{2}\gamma  +\left(k^2-\frac{|\gamma|^2}{4}\right)^{1/2} d\in \mathbb{R}^3,\\
\xi^{(2)} &=-\frac{1}{2}\gamma -\left(k^2-\frac{|\gamma|^2}{4}\right)^{1/2} d\in \mathbb{R}^3.
\end{align*}
It is easy to verify that $\xi^{(1)}+\xi^{(2)} = -\gamma$, $\xi^{(1)}\cdot \xi^{(1)}= \xi^{(2)}\cdot \xi^{(2)} = k^2$. Let 
\[
U_i(x,\xi^{(i)}) = e^{\ii  \xi^{(i)}\cdot x},\quad i=1,2,
\]
which satisfies $\Delta U_i +k^2 U_i=0$.  

It follows from the It\^{o} isometry that 
 \begin{align*}
\E[\langle f, U_1\rangle \langle f, U_2\rangle]  &=  \int_{\mathbb{R}^3} \int_{\mathbb{R}^3} U_1(x) U_2(y) \E [f(x)f(y)]dxdy\\
&=     \int_{\mathbb{R}^3} \mu(x) e^{-\ii\gamma\cdot x}dx=(2\pi)^3\hat{\mu}(\gamma).
\end{align*}
Using Lemma \ref{lem:Ewhitehomo}, we obtain 
\begin{align*}
 |\hat{\mu}(\gamma) |&\leq  c_2 (2\pi)^{-3} M \Vert U_1\Vert_{L^2(B_{\hat R})} \Vert U_2\Vert_{L^2(B_{\hat R})}\\
 &\leq \frac{\hat R^3}{6\pi^2}  c_2 M.
\end{align*}
The proof is completed by taking $c_4 = \frac{\hat R^3}{6\pi^2}  c_2$. 
\end{proof}

We are now in a position to prove Theorem  \ref{thm:white_homo}.

\begin{proof}
It is clear to note that 
\begin{align*}
\Vert \mu\Vert_{L^\infty(B_1)} &\leq \sup_{x\in B_1}\bigg|\int_{\mathbb{R}^3} e^{\ii \gamma\cdot x} \hat{\mu}(\gamma)d\gamma\bigg|\\
&\leq \int_{|\gamma|\leq \rho} |\hat{\mu}(\gamma)|d\gamma+ \int_{|\gamma|>\rho} |\hat{\mu}(\gamma)|d\gamma\\
&:=I_1(\rho) + I_2(\rho).
\end{align*}

By Lemma \ref{lem:Fourierlow_white_homo}, we have 
\begin{eqnarray*}
I_1(\rho)\leq \frac{4\pi \rho^3}{3} c_4 M.
\end{eqnarray*}
Combining the above estimate and \eqref{eqn:fourierhigh} yields 
\begin{eqnarray*}
\Vert \mu\Vert_{L^\infty(B_1)} &\leq&\frac{4\pi \rho^3}{3} c_4 M+c_1\frac{1}{\rho^{s-3}}.
\end{eqnarray*}
For simplicity, we take $\hat R=2R$ and $\epsilon = \frac{\hat R-R}{4} = \frac{R}{4}$. Now for $k>\frac{2}{R-1}$, we have the radius $r = \frac{1}{k}$ in \eqref{est:U1} and \eqref{est:U2}. If $k$ also satisfies $k>\frac{4\sqrt{c_3}}{R}$, then we can choose a constant  $c_5$ depending only on $s$ and $R$ such that  $c_4= c_5 k^5$. This follows the observation that  $c_4 = \frac{(\hat R)^3}{6\pi^2}c_2$ and $c_2=12 R^2 r^{-3} \left(1+ \left(k^2 +\frac{c_3}{\epsilon^2}\right)^{1/2}\right)^2$ in the proof of Lemma \ref{lem:Ewhitehomo}. Taking $k>\max\{\frac{2}{R-1},\frac{4\sqrt{c_3}}{R},(\frac{1}{2})^{\frac{m}{m+5}} (\frac{3c_1}{4\pi c_5 M})^{\frac{1}{m+5}}\} $ and  $\rho = (\frac{3c_1}{4\pi c_4 M})^{\frac{1}{m}}$, one can verify that $\rho<2k$,  and there holds
\begin{equation*}
\Vert \mu\Vert_{L^\infty(B_1)}\leq \left(\frac{4\pi }{3} c_4\right)^{1- \frac{3}{s}}c_1^{\frac{3}{s}} M^{1-\frac{3}{s}}.
\end{equation*}
Taking  $C_1 = (\frac{4\pi }{3} c_4)^{1- \frac{3}{s}}c_1^{\frac{3}{s}}$ yields the desired estimate \eqref{est:homowhite}. 

Recalling the definitions $C_1 = (\frac{4\pi}{3} c_4)^{1- \frac{3}{s}}c_1^{\frac{3}{s}}$ and $c_4 = c_5 k^5$, where $c_5$ and $c_1$ depend only on $s$, $R$ and $s$, respectively, we can deduce that $C_1 = O(k^{5(1-\frac{3}{s})})$ as $k\to\infty$.
\end{proof}

It is worth noting that the estimate in Theorem \ref{thm:white_homo} can be adapted for the two-dimensional problem with a moderate adjustment to the proof, which is given as follows for completeness. 

For the random source $f$ with a strength $\mu$, by the Paley--Wiener--Schwartz theorem, we have 
\[
|\hat{\mu}(\gamma)|\leq\overline{c}_0 (1+|\gamma|)^{-s}. 
\]
Thus, there exists a positive constant $\overline{c}_1$ depending on $s$ such that  
\begin{equation*}%\label{eqn:fourierhigh2d}
\int_{|\gamma|>\rho} |\hat{\mu}(\gamma)|d\gamma\leq \overline{c}_1 \rho^{2-s}.
\end{equation*}
For the Fourier coefficients of the strength $\mu$ in low frequency modes, there still exists a positive constant $ \overline{c}_4 $ such that
\[
|\hat{\mu}(\gamma)|\leq \overline{c}_4 M,\quad |\gamma|<2k.
\]
Combining the above estimates gives
\begin{align*}
\Vert \mu\Vert_{L^\infty(B_1)} &\leq \sup_{x\in B_1}\bigg|\int_{\mathbb{R}^3} e^{\ii \gamma\cdot x} \hat{\mu}(\gamma)d\gamma\bigg|\\
&\leq \int_{|\gamma|\leq \rho} |\hat{\mu}(\gamma)|d\gamma+ \int_{|\gamma|>\rho} |\hat{\mu}(\gamma)|d\gamma\\
&\leq 4\pi\rho^2 \overline{c}_4 M+\overline{c}_1 \rho^{2-s}.
\end{align*}
Taking $\rho =\left( \frac{\overline{c}_1}{4\pi \overline{c}_4 M}\right)^{\frac{1}{s}}$, we obtain for $\rho\leq 2k$ that 
\begin{equation*}
\Vert \mu\Vert_{L^\infty(B_1)} \leq \overline{C}_1 M^{1-\frac{2}{s}},
\end{equation*}
where 
\[
 \overline{C}_1=\overline{c}_1^{\frac{2}{s}} ( 4\pi \overline{c}_4)^{1-\frac{2}{s}}. 
\]

\subsection{Inhomogeneous media}

We now investigate the stability of the inverse random source problem for the Helmholtz equation in an inhomogeneous medium. In this analysis, we consider the use of complex geometric optics (CGO) solutions as special solutions.

Similarly, we initiate our analysis by deriving an estimate that establishes a connection between the internal distribution of the random source and the measurement of the wave field $u$ on $\partial B_R$.

\begin{lemma}%\label{lem:Ewhiteinhomo}
Consider two solutions $U_i\in C^3(\overline{B_{\hat R}})$, $i=1,2$, to the equation $\Delta u + k^2 (1+q) u = 0$, where $\hat R > R > 1$. Let $f = \sqrt{\mu}(x) \dot{W}$. Then there exists a positive constant $c_2$ depending on $k$, $R$, and $\hat R$ such that the following estimate holds:
\begin{equation}\label{eq:estfU_inhomo}
|\E[\langle f, U_1\rangle \langle f, U_2\rangle]| \leq c_2 M \Vert U_1\Vert_{L^2(B_{\hat R})} \Vert U_2\Vert_{L^2(B_{\hat R})}.
\end{equation}
\end{lemma}

\begin{proof}
Using a similar line of reasoning as in Lemma \ref{lem:Ewhitehomo}, we denote $ \sqrt{\mu}(x)\dot{W}_x^N$ as $f^N$, and let $u^N$ represent the solution to $\Delta u^N + k^2 (1+q) u^N = f^N$ satisfying the Sommerfeld radiation condition. By employing this approach, we deduce for $U_i$, $i=1,2,$ that 
\[
 \langle f^N, U_i\rangle  = \int_{\partial B_R} \left( \partial_\nu u^N(x) U_i(x) -\partial_\nu U_i(x) u^N(x)\right) ds(x)
\]
and 
\begin{align*}
&\mathbb{E}[\langle f^N, U_1\rangle \langle f^N, U_2\rangle]\\
&\leq  M^N\Vert U_1\Vert_{L^2(\partial B_R)}\Vert U_2\Vert_{L^2(\partial B_R)}+M^N\Vert U_1\Vert_{L^2(\partial B_R)}\Vert \partial_\nu U_2\Vert_{L^2(\partial B_R)}\\
&\quad +M^N\Vert \partial_\nu U_1\Vert_{L^2(\partial B_R)}\Vert  U_2\Vert_{L^2(\partial B_R)}+M^N\Vert \partial_\nu U_1\Vert_{L^2(\partial B_R)}\Vert \partial_\nu U_2\Vert_{L^2(\partial B_R)},
\end{align*}
where 
\[
M^N = \max\left\{ \Vert F^N_1\Vert_{L^2(\partial B_R\times \partial B_R)},\, \Vert F^N_2\Vert_{L^2(\partial B_R\times \partial B_R)},\, \Vert F^N_3\Vert_{L^2(\partial B_R\times \partial B_R)}\right\}
\]
and
\begin{align*}
F^N_1(x,y)&=\E[u^N(x) u^N(y)],\\
F^N_2(x,y) &=\E[u^N(x)\partial_\nu u^N(y)],\\
F^N_3(x,y)&=\E[\partial_\nu u^N(x) \partial_\nu u^N(y)] .
\end{align*}

Given that $U_i$ satisfies $\Delta u + k^2 u = 0$ in $B_{\hat R}\setminus  \overline{B_1}$, we can conclude that the same type of inequalities as \eqref{est:U1} and \eqref{est:U2} apply to $U_i$. Thus there exists a positive constant $c_2$ such that
\[
\E[\langle f^N, U_1\rangle \langle f^N, U_2\rangle]\leq c_2 M^N \Vert U_1\Vert_{L^2(B_{\hat R})} \Vert U_2\Vert_{L^2(B_{\hat R})}.
\]

Next, we establish \eqref{eq:estfU_inhomo} by considering a sequence of approximations ${f^N}$, where $f^N = \sqrt{\mu}(x)\dot{W}_x^N$. The triangulation $\mathcal{T}_N$ is chosen such that $\max_{K_j\in \mathcal{T}_N}\text{diam}(K_j)\to 0$ as $N\to \infty$. By employing the same arguments as in Lemma \ref{lem:Ewhitehomo}, we can show that
\begin{equation*}%\label{est:EEinhomo}
\E[\langle f^N, U_1\rangle \langle f^N, U_2\rangle]\to \E[\langle f, U_1\rangle \langle f, U_2\rangle]\quad \text{as }N\to\infty.  
\end{equation*}

The last step is to verify that $M^N\to M$ as $N\to \infty$, which can be accomplished by checking the convergence $\Vert F^N_j\Vert_{L^2(\partial B_R\times \partial B_R)}\to \Vert F_j\Vert_{L^2(\partial B_R\times \partial B_R)}, j=1, 2, 3$ as $N\to\infty$.  Using \eqref{eq:repre_inhomo}, we have 
\begin{align*}
\E[u^N(x)u^N(y)]&=\E\Big[\left(-k^2 \mathcal{G}[qu^N](x) + \mathcal{G}[\sqrt{\mu}\dot{W}^N](x)\right) \nonumber\\
&\qquad\qquad\times\left(-k^2 \mathcal{G}[qu^N](y) + \mathcal{G}[\sqrt{\mu}\dot{W}^N](y)\right)\Big]\\
&=k^4  \E\left[ \mathcal{G}[qu^N](x) \mathcal{G}[qu^N](y)\right]-k^2 \E\left[ \mathcal{G}[qu^N](x)\mathcal{G}[\sqrt{\mu}\dot{W}^N](y)\right]\\
&\quad -k^2 \E\left[\mathcal{G}[\sqrt{\mu}\dot{W}^N](x) \mathcal{G}[qu^N](y)\right]\\
&\quad +\E\left[\int_{B_1} G(x,z_1)\sqrt{\mu}(z_1)dW_{z_1}^N  \int_{B_1} G(y,z_2)\sqrt{\mu}(z_2)dW_{z_2}^N\right]\\
&:= k^4 I_1^N(x,y)+ k^2 I_2^N(x,y)+ k^2 I_2^N(y,x)+ I_3^N(x,y)
\end{align*}
and
\begin{align*} 
\E[u(x)u(y)] &= \E\left[\left(-k^2 \mathcal{G}[qu](x) + \mathcal{G}[\sqrt{\mu}\dot{W}](x)\right)\left(-k^2 \mathcal{G}[qu](y) + \mathcal{G}[\sqrt{\mu}\dot{W}](y)\right)\right]\\
&=k^4  \E\left[ \mathcal{G}[qu](x) \mathcal{G}[qu](y)\right] -k^2 \E\left[ \mathcal{G}[qu](x)\mathcal{G}[\sqrt{\mu}\dot{W}](y)\right]\\
&\quad -k^2 \E\left[\mathcal{G}[\sqrt{\mu}\dot{W}](x) \mathcal{G}[qu](y)\right]\\
&\quad +\E\left[\int_{B_1} G(x,z_1)\sqrt{\mu}(z_1)dW_{z_1}  \int_{B_1} G(y,z_2)\sqrt{\mu}(z_2)dW_{z_2}\right]\\
&:= k^4 I_1(x,y)+ k^2 I_2(x,y)+ k^2 I_2(y,x)+ I_3(x,y).
\end{align*}

First we consider $I_1^N-I_1$. After straightforward calculations, we obtain
\begin{align}
 &\E\left[ \mathcal{G}[qu^N](x) \mathcal{G}[qu^ N](y)\right]-\E\left[ \mathcal{G}[qu](x) \mathcal{G}[qu](y)\right] \nonumber\\
 &=\int_{B_1} \int_{B_1} G(x,z_1)G(y, z_2)  q(z_1)q(z_2) \left(\E[u^N(z_1)u^N(z_2)]-\E[u(z_1)u(z_2)] \right)dz_1 dz_2 \nonumber\\
 &=\int_{B_1} \int_{B_1} G(x,z_1)G(y, z_2) q(z_1) q(z_2)\E[u^N(z_1)(u^N(z_2)-u(z_2))]dz_1 dz_2 \nonumber\\
 &\quad +\int_{B_1} \int_{B_1} G(x,z_1)G(y, z_2)  q(z_1)q(z_2)\E[(u^N(z_1)-u(z_1))u(z_2)] dz_1 dz_2 \nonumber\\
 &= \E\left[\int_{B_1} G(x,z_1) q(z_1)  u^N(z_1) dz_1  \int_{B_1}G(y, z_2) q(z_2)\left(u^N(z_2)-u(z_2)\right)dz_2\right] \nonumber\\
 &\quad+ \E\left[\int_{B_1} G(x,z_1) q(z_1)  (u^N(z_1)-u(z_1)) dz_1  \int_{B_1}G(y, z_2) q(z_2)u(z_2)dz_2\right] \nonumber\\
 &\leq\Vert G(x,\cdot) q(\cdot)  \Vert_{L^2(B_1)}\Vert G(y,\cdot) q(\cdot)  \Vert_{L^2(B_1)}\E[\Vert  u^N\Vert_{L^2(B_1)} \Vert  u^N-u\Vert_{L^2(B_1)} ] \nonumber\\
 &\quad +\Vert G(x,\cdot) q(\cdot)  \Vert_{L^2(B_1)}\Vert G(y,\cdot) q(\cdot)  \Vert_{L^2(B_1)}\E[\Vert  u\Vert_{L^2(B_1)} \Vert  u^N-u\Vert_{L^2(B_1)}].\label{eq:estI1}
\end{align}
By observing that $\Vert G(x,\cdot) q(\cdot) \Vert_{L^2(B_1)}$ is uniformly bounded above by a constant $C$ that depends on $R$ for all $x\in \partial B_R$, we can conclude from the proof of Lemma \ref{lem:unique} that $\E[\Vert u^N-u\Vert^2_{L^2(B_1)}]\to 0$ as $N\to\infty$. Hence, for $N$ exceeding a positive constant $N_0$, we have $\E[\Vert u^N\Vert_{L^2(B_1)}] \leq 2\E[\Vert u\Vert_{L^2(B_1)}]$. Consequently, utilizing \eqref{eq:estI1} yields 
\begin{align*}
 &\E\left[ \mathcal{G}[qu^N](x) \mathcal{G}[qu^ N](y)\right]-\E\left[ \mathcal{G}[qu](x) \mathcal{G}[qu](y)\right]\\
 &\leq 3 C^2 \E[\Vert u\Vert^2_{L^2(B_1)}]^{\frac{1}{2}}  \E[\Vert u_N-u\Vert^2_{L^2(B_1)} ]^{\frac{1}{2}}\to 0\quad\text{as } N\to\infty, 
\end{align*}
which shows
\begin{equation*}%\label{est:I1}
\Vert I_1^N-I_1\Vert_{L^2(\partial B_R\times \partial B_R)} \to 0 \quad\text{as } N\to\infty. 
\end{equation*}

We proceed to prove the convergence of $I_2^N$. Let $u_0(x)=\mathcal{G}[\sqrt{\mu}\dot{W}](x)$ and $u_0^N(x)=\mathcal{G}[\sqrt{\mu}\dot{W}^N](x)$. It follows from the proof of Lemma \ref{lem:Ewhitehomo} that $\E[u^N_0(x) u^N_0(y)]\to \E[u_0(x) u_0(y)]$ as $N\to \infty$. It is easy to verify that 
\begin{align*}
& \E\left[ \mathcal{G}[qu](x)\mathcal{G}[\sqrt{\mu}\dot{W}](y)\right]- \E\left[ \mathcal{G}[qu^N](x)\mathcal{G}[\sqrt{\mu}\dot{W}^N](y)\right]\\
&= \E\left[\int_{B_1} G(x,z_1) q(z_1)  u(z_1) dz_1  \int_{B_1}G(y, z_2)\sqrt{\mu}(z_2)(dW_{z_2}-dW^N_{z_2})\right]\\
 &\quad +\E\left[\int_{B_1} G(x,z_1) q(z_1)  \left(u^N(z_1)-u(z_1)\right) dz_1  \int_{B_1}G(y, z_2)\sqrt{\mu}(z_2)dW^N_{z_2}\right]\\
 &= \E\left[\int_{B_1} G(x,z_1) q(z_1)  u(z_1) dz_1\left( u_0^N(y)-u_0(y)\right)\right]\\
 &\quad +\E\left[\int_{B_1} G(x,z_1) q(z_1)  \left(u^N(z_1)-u(z_1)\right) dz_1 u_0^N(y)\right]\\
 &\leq \Vert G(x,\cdot ) q(\cdot ) \Vert_{L^2(B_1)} \left( \E[ \Vert u\Vert_{L^2(B_1)}  | u_0^N(y)-u_0(y)|]+\E[ \Vert u^N-u \Vert_{L^2(B_1)} |u_0^N(y)|]\right),
\end{align*}
where we have used the Cauchy–Schwarz inequality for the last line.
As  there exists a uniform upper bound of $\Vert G(x,\cdot) q(\cdot) \Vert_{L^2(B_1)}$  
 for all $x\in \partial B_R$, we can find a constant $C$ such that
\begin{align*}
&\Vert I_2^N-I_2\Vert_{L^2(\partial B_R\times \partial B_R)}  \\
&\leq \int_{\partial B_R}\Vert G(x,\cdot ) q(\cdot) \Vert_{L^2(B_1)} ds(x) \int_{\partial B_R} \E\left[ \Vert u\Vert_{L^2(B_1)}   | u_0^N(y)-u_0(y)|\right] ds(y)\\
 &\quad + \int_{\partial B_R}\Vert G(x,\cdot ) q(\cdot) \Vert_{L^2(B_1)} ds(x) \int_{\partial B_R}\E\left[ \Vert u^N-u \Vert_{L^2(B_1)} |u_0^N(y)|\right]ds(y)\\
 &\leq C\E\left[ \Vert u\Vert_{L^2(B_1)}  \int_{\partial B_R}  | u_0^N(y)-u_0(y)| ds(y)\right]  \nonumber\\
&\quad+ C\E\left[ \Vert u^N-u \Vert_{L^2(B_1)}  \int_{\partial B_R}|u_0^N(y)|ds(y)\right]\\
 &\leq C \E\left[ \Vert u\Vert_{L^2(B_1)}^2\right]^{\frac{1}{2}} \E\left[ \left(\int_{\partial B_R}  | u_0^N(y)-u_0(y)| ds(y)\right)^2\right]^{\frac{1}{2}}  \\
 &\quad + C \E[ \Vert u^N-u \Vert_{L^2(B_1)}^2]^{\frac{1}{2}} \E\left[\left( \int_{\partial B_R}|u_0^N(y)|ds(y)\right)^2\right]^{\frac{1}{2}}  \\
 &:= I_{2,1}+ I_{2,2},
\end{align*}
where we have used the Cauchy--Schwarz inequality for the random variables.

Note that 
\begin{align*}
\E\left[ \left(\int_{\partial B_R}  | u_0^N(y)-u_0(y)| ds(y)\right)^2\right]&\leq \E\left[ \left( 4\pi R^2 \int_{\partial B_R}  | u_0^N(y)-u_0(y)|^2 ds(y)\right)\right]\\
&\leq 4\pi R^2   \E\left[ \Vert u_0^N-u_0\Vert_{L^2(\partial B_R)}^2 \right].
\end{align*}
As  $u_0$ satisfies the Helmholtz equation $\Delta u+k^2 u =f $, it is proved in Lemma \ref{lem:boundary} that $u_0$ is analytic in $B_{\hat R}\setminus \overline{B_1}$. It can be observed that  $u_0^N$ is analytic in $B_{\hat R}\setminus \overline{B_1}$, which implies that $u_0^N-u_0 \in C^{\infty}(B_{\hat R}\setminus \overline{B_1})$ and satisfies $\Delta u +k^2 u = 0$ in $B_{\hat R}\setminus \overline{B_1}$. We can apply the mean value property for solutions to the Helmholtz equation and deduce that there exists a constant $C$ depending on  $R$, $\hat R$, and $k$ such that 
\begin{equation*}
\Vert u_0^N-u_0\Vert_{L^2(\partial B_R)}\leq C \Vert u_0^N-u_0\Vert_{L^2( B_{\hat R})}.
\end{equation*}
Since $\E [\Vert u_0^N-u_0\Vert_{L^2( B_{\hat R})}]\to 0$ as $N\to \infty$, we have $I_{2,1}\to 0$. 
Using the Cauchy--Schwartz inequality and the mean value property, we can also claim that the term $\E\left[\left( \int_{\partial B_R}|u_0^N(y)|ds(y)\right)^2\right]  $ has a uniform upper bound for $N\geq N_0$.
Combining with $\E [\Vert u^N-u\Vert_{L^2( B_{\hat R})}]\to 0$ leads to $I_{2,2}\to 0$  as $N\to \infty$. The convergence of $I^N_3$ follows from the proof of Lemma \ref{lem:Ewhitehomo}. Hence we can conclude that 
\[
\Vert F^N_1\Vert_{L^2(\partial B_R\times \partial B_R)}\to \Vert F_1\Vert_{L^2(\partial B_R\times \partial B_R)}\quad\text{as }N\to\infty. 
\]

Next is to show the convergence of $F_2^N$. It is clear to note that 
\begin{align*}
&\E[u^N(x)\partial_\nu u^N(y)] - \E[u(x)\partial_\nu u(y)] \\
&= k^4  \left(\E\left[ \mathcal{G}[qu^N](x)\partial_\nu\left(\mathcal{G}[qu^N]\right)(y)-\mathcal{G}[qu](x)\partial_\nu\left(\mathcal{G}[qu]\right)(y)\right]\right)\\
&\quad -k^2 \left(\E\left[ \mathcal{G}[qu^N](x)\partial_\nu u_0^N (y)- \mathcal{G}[qu](x)\partial_\nu u_0(y)\right]\right)\\
&\quad -k^2 \left(\E\left[u_0^N(x) \partial_\nu\left(\mathcal{G}[qu^N]\right)(y)-u_0(x) \partial_\nu\left(\mathcal{G}[qu]\right)(y)\right]\right)\\
&\quad +\left(\E[u_0^N(x) \partial_\nu u_0^N(y)] -\E[u_0(x) \partial_\nu u_0(y)]\right).
\end{align*}
According to the formula for the Green's function, $\Vert \partial_\nu G(x,\cdot) q(\cdot) \Vert_{L^2(B_1)}$ is bounded above by a constant $C$ that depends on $k$ and $R$ for all $x\in \partial B_R$. By Lemma \ref{lem:boundary}, we may also show that the first order derivatives of $ u_0^N-u_0$ are analytic in $B_{\hat R}\setminus\overline{B_1}$. It follows from  the mean value property of the solution to the Helmholtz equation and  the Caccioppoli's inequality that there exists a constant, denoted by $C$, which depends on $k$, $R$, and $\hat R$, such that
\begin{eqnarray*}
\Vert \partial_\nu u_0^N-\partial_\nu u_0\Vert_{L^2(\partial B_R)}\leq C \Vert u_0^N-u_0\Vert_{L^2( B_{\hat R})}.
\end{eqnarray*}
By adopting a similar argument to the proof of the convergence of $F_1^N$, we can show that $F_2^N\to F_2$ as $N\to \infty$.

We are now left with considering $F_3^N$. It follows from the expressions of the solutions $u$ and $u^N$ that
\begin{align*}
&\E[\partial_\nu u^N(x)\partial_\nu u^N(y)] - \E[\partial_\nu u(x)\partial_\nu u(y)] \\
&=k^4  (\E[\partial_\nu  \mathcal{G}[qu^N](x)\partial_\nu(\mathcal{G}[qu^N])(y)-\partial_\nu \mathcal{G}[qu](x)\partial_\nu(\mathcal{G}[qu])(y)])\\
&\quad -k^2 (\E[ \partial_\nu \mathcal{G}[qu^N](x)\partial_\nu u_0^N (y)- \partial_\nu \mathcal{G}[qu](x)\partial_\nu u_0(y)])\\
&\quad -k^2 (\E[\partial_\nu u_0^N(x) \partial_\nu(\mathcal{G}[qu^N])(y)-\partial_\nu u_0(x) \partial_\nu(\mathcal{G}[qu])(y)])\\
&\quad +\E[\partial_\nu u_0^N(x) \partial_\nu(u_0^N(y))] -\E[\partial_\nu u_0(x) \partial_\nu u_0(y)].
\end{align*}
Then the convergence of $F_3^N$ can be established by utilizing the boundedness of $\Vert \partial_\nu G(x,\cdot) q(\cdot) \Vert_{L^2(B_1)}$ and the convergence of $u_0^N$ and $u^N$, which completes the proof. 
\end{proof}

Next, we establish the stability estimate for the reconstruction of the strength $\mu$ by estimating the Fourier coefficients $\hat{\mu}$. To achieve this, we utilize special solutions to the equation
\begin{equation*}%\label{eqn:inhomo}
\Delta u+ k^2 (1+q) u = 0.
\end{equation*}
The solutions, known as complex geometric optics (CGO) solutions, possess the structure described in the following lemma.

\begin{lemma}\label{lem:CGO}
 For all $\xi\in\mathbb{C}^3$ satisfying $\xi\cdot\xi=k^2$ and $|\Im (\xi)|\geq c_5$, where $c_5$ depends on  $s$, $k$, and $\hat R$, there exists a function $v(\cdot,  \xi)\in C^3(B_{\hat R})$ such that 
\[
U(x,\xi) =e^{\ii \xi\cdot x}(1+v(x,\xi)),\ x\in B_{\hat R}
\]
satisfies $\Delta U+k^2 (1+q)U = 0$ in $B_{\hat R}$, and the following estimate holds:
\[
\Vert v(\cdot,\xi)\Vert_{L^2(B_{\hat R})}\leq \frac{c_6}{|\Im(\xi)|},
\]
where $c_6$ is a positive constant depending on $k$ and $\|q\|_{L^\infty(B_1)}$. 
\end{lemma}

\begin{proof}
It is proved in \cite[Lemma 2.9]{hahner1998acoustic} that when $q\in C^{0,\alpha}(\overline{B_1}), 0<\alpha<1$, there exists a function $v(\cdot,  \xi)\in C^2(B_{\hat R})$ such that $U(x,\xi) =e^{\ii \xi\cdot x}(1+v(x,\xi))$ satisfies $\Delta U+k^2 (1+q)U = 0$ in $B_{\hat R}$. Here we sketch the proof to show that when $q\in C^{1,\alpha}(\overline{B_1})$, the constructed function $v(\cdot,  \xi)$ has a higher regularity and belongs to $C^3(B_{\hat R})$.

As in the proof of \cite[Lemma 2.9]{hahner1998acoustic}, we take $v\in C^2(B_{\hat R})$ as a solution 
to the integral equation
\begin{equation}\label{eqn:defv}
v(\cdot,\xi) = -k^2 G_\xi (qv(\cdot,\xi))-k^2 G_\xi(q),
\end{equation}
where the operator $G_\xi$ is defined in \cite[Theorem 2.8]{hahner1998acoustic} and has the following properties:
\begin{enumerate}
\item If $f\in C^{0,\alpha} (\overline{B_{\hat R}}) $, then $G_\xi f\in C^2 (\overline{B_{\hat R}})$;  
\item  If $f\in C_0^1 (B_{\hat R}) $, then $\partial_j(G_\xi f) = G_\xi(\partial_j f)$.
\end{enumerate}
It can be verified that $U(x,\xi) =e^{\ii \xi\cdot x}(1+v(x,\xi))$ satisfies the Helmholtz equation $\Delta U+k^2 (1+q)U = 0$ in $B_{\hat R}$. Recall that $q\in C^{1,\alpha}(\overline{B_1})$ and is compactly supported in $B_1$. Clearly, we have $q\in C^{1,\alpha}(\overline{B_{\hat R}})$ and both $q$ and $qv(\cdot,\xi)$ belong to $C_0^1 (B_{\hat R}) $. Taking the partial derivative on both sides of \eqref{eqn:defv} and using the property of $G_\xi$ yields
\begin{align*}
\partial_j v(\cdot,\xi) &= -k^2 \partial_j (G_\xi (qv(\cdot,\xi)))-k^2 \partial_j( G_\xi(q))\\
&=-k^2 G_\xi ( \partial_j (qv(\cdot,\xi)))-k^2  G_\xi( \partial_j (q)).
\end{align*}
Since $\partial_j (qv(\cdot,\xi)),\partial_j (q) \in  C^{0,\alpha} (\overline{B_{\hat R}})$, it follows from the property of $G_\xi$ that $\partial_j v(\cdot,\xi)\in C^2(\overline{B_{\hat R}})$. Hence we arrive at $v(\cdot,\xi)\in C^3(\overline{B_{\hat R}})$.
\end{proof}
\begin{remark}
By following the proof in \cite[Lemma 2.9]{hahner1998acoustic}, we can establish that 
the constants $c_5$ and $c_6$ in Lemma \ref{lem:CGO} satisfy $c_5 = O(k^2)$ and $c_6=O(k^2)$ as $k\to\infty$. 
\end{remark}

Using the special solutions, we can estimate the Fourier coefficients $\hat\mu(\gamma)$ for $|\gamma|\leq \rho$.

\begin{lemma}\label{lem:Fourierlow}
Assume $\rho\geq 2$ and define $t_0 = \sqrt{c_5^2 + \rho^2}$, where $c_5$ is the constant given in Lemma \ref{lem:CGO} for $\hat R= 2R$. Then there exists a positive constant $c_7$ depending on $k$ and $R$ such that the following estimate holds for all $\gamma\in\mathbb{R}^3$ with $|\gamma|< \rho$ and for all $t>t_0$:
\[
|\hat{\mu}(\gamma)|\leq c_7\left(M e^{4Rt}+\frac{\Vert \mu\Vert_{L^\infty(B_1)} }{t}\right).
\]
\end{lemma}

\begin{proof}
Let $\hat R=2R$. For $t>t_0$ and a fixed $\gamma\in \mathbb{R}^3$, we choose two unit vectors $d_1$ and $d_2$ such that $d_1\cdot d_2 = d_1\cdot \gamma= d_2\cdot\gamma = 0$. We define two complex vectors as follows: 
\begin{align}
\xi_t^{(1)} &= -\frac{1}{2}\gamma + \ii t d_1 +\left(t^2+k^2-\frac{|\gamma|^2}{4}\right)^{1/2} d_2\in \mathbb{C}^3, \label{eq:xit1}\\
\xi_t^{(2)} &= -\frac{1}{2}\gamma - \ii t d_1 -\left(t^2+k^2-\frac{|\gamma|^2}{4}\right)^{1/2} d_2\in \mathbb{C}^3.\label{eq:xit2}
\end{align}
It is easy to verify that $\xi_t^{(1)}+\xi_t^{(2)} = -\gamma$ and $\xi_t^{(1)}\cdot \xi_t^{(1)}= \xi_t^{(2)}\cdot \xi_t^{(2)} = k^2$. By Lemma \ref{lem:CGO}, there exist two geometric optics solutions 
\begin{eqnarray*}
U_i(x,\xi_t^{(i)}) &=& e^{\ii  \xi_t^{(i)}\cdot x}\left(1+v_i(x,\xi_t^{(i)})\right),\quad i=1,2,
\end{eqnarray*}
satisfying the Helmholtz equation $\Delta u +k^2 (1+q)u = 0$ for $x\in B_{2R}$, and the solutions have the property that $\Vert v_i\Vert_{L^2(B_{2 R})}\leq \frac{c_6}{|\Im \xi_t^{(i)}|}$. For simplicity, we denote $U_i(x,\xi_t^{(i)})$ by $U_i(x), i=1,2$. 

A simple calculation yields 
\begin{equation*}
U_1(x)U_2(x)  = e^{-\ii\gamma\cdot x}(1+p(x,t)),
\end{equation*}
where 
\[
p(x,t) = v_1(x,\xi_t^{(1)}) + v_2(x,\xi_t^{(2)}) + v_1(x,\xi_t^{(1)})v_2(x,\xi_t^{(2)}).
\]
It can be verified that the function $p(x,t)$ satisfies
\begin{align*}
\Vert p(x,t)\Vert_{L^1(B_1)} &= \int_{B_1} \left| v_1(x,\xi_t^{(1)}) + v_2(x,\xi_t^{(2)}) + v_1(x,\xi_t^{(1)})v_2(x,\xi_t^{(2)})\right|dx \nonumber \\
&\leq  \sqrt{\frac{4\pi}{3}}\Vert v_1\Vert_{L^2(B_1)} +\sqrt{\frac{4\pi}{3}} \Vert v_2\Vert_{L^2(B_1)}  +\Vert v_1\Vert_{L^2(B_1)}  \Vert v_2\Vert_{L^2(B_1)}   \nonumber \\
&\leq \sqrt{\frac{4\pi}{3}} \frac{c_6}{|\Im \xi_t^{(1)}|} + \sqrt{\frac{4\pi}{3}}\frac{c_6}{|\Im \xi_t^{(2)}|} + \frac{c_6^2}{|\Im \xi_t^{(1)}| |\Im \xi_t^{(2)}|}\nonumber \\
&\leq\left(2\sqrt{\frac{4\pi}{3}}c_6 +c_6^2\right)\frac{1}{t}.%\label{eq:estp}
\end{align*}
Using the  It\^{o} isometry, we obtain 
\begin{align*}
\E[\langle f, U_1\rangle \langle f, U_2\rangle]  &= \E\left[\int_{\mathbb{R}^3} \int_{\mathbb{R}^3} \sqrt{\mu}(x)U_1(x) \sqrt{\mu}(y) U_2(y)  \dot{W} (x) \dot {W}(y)dxdy\right] \nonumber\\
&=\int_{\mathbb{R}^3}  \mu(x)U_1(x) U_2(x)dx \nonumber\\\
&= \int_{\mathbb{R}^3}  \mu(x) e^{-\ii\gamma\cdot x}(1+p(x,t)))dx  \nonumber\\\
&=(2\pi)^3\hat{\mu}(\gamma) +  \int_{\mathbb{R}^3}  \mu(x) e^{-\ii\gamma\cdot x}p(x,t))dx.
\end{align*}
We have from straightforward calculations that 
\begin{align*}
\Vert U_i\Vert_{L^2(B_{\hat R})}&= \Vert e^{\ii \xi_t^{(i)}\cdot x}(1+v_i(x,\xi_t^{(i)}))\Vert_{L^2(B_{\hat R})}\\
&\leq\Vert e^{\ii\xi_t^{(i)}\cdot x}\Vert_{L^\infty(B_{\hat R})} \Vert 1+v_i\Vert_{L^2(B_{\hat R})}\\
&\leq  \left(4\sqrt{\frac{2\pi R^3}{3}}+\frac{c_6}{t} \right)e^{2Rt}.
\end{align*}
Combining the above estimates shows that there exists a positive constant $c_7$ depending on $k$ and $R$ such that
\begin{align*}
|\hat{\mu}(\gamma)| &\leq (2\pi)^{-3}\left( | \E[\langle f, U_1\rangle \langle f, U_2\rangle]| + \left|\int_{B_1}  \mu(x) e^{-\ii\gamma\cdot x}p(x,t))dx\right|\right)\\
&\leq (2\pi)^{-3} c_3 M \Vert U_1\Vert_{L^2(B_{\hat R})} \Vert U_2\Vert_{L^2(B_{ \hat R})} + (2\pi)^{-3} \left(2\sqrt{\frac{4\pi}{3}}c_6 +c_6^2\right)\Vert \mu\Vert_{L^\infty(B_1)} \frac{1}{t}\\
&\leq c_7\left(M e^{4Rt}+\frac{\Vert \mu\Vert_{L^\infty(B_1)} }{t}\right),
\end{align*}
which completes the proof. 
\end{proof}

Now we are ready to prove Theorem \ref{thm:white_inhomo}.

\begin{proof}
It is clear to note that 
\begin{align*}
\Vert \mu\Vert_{L^\infty(B_1)} &\leq \sup_{x\in B_1}\left|\int_{\mathbb{R}^3} e^{\ii \gamma\cdot x} \hat{\mu}(\gamma)d\gamma\right|\\
&\leq \int_{|\gamma|\leq \rho} |\hat{\mu}(\gamma)|d\gamma+ \int_{|\gamma|>\rho} |\hat{\mu}(\gamma)|d\gamma\\
&:= I_1(\rho) + I_2(\rho).
\end{align*}
For $t>t_0$, it follows from Lemma \ref{lem:Fourierlow} that 
\begin{equation*}
I_1(\rho)\leq \frac{4\pi \rho^3}{3}  c_7\left( M e^{4Rt}+ \frac{\Vert \mu\Vert_{L^\infty(B_1)} }{t}\right).
\end{equation*}
Denoting $\frac{4\pi}{3}c_7$ still by  $c_7$, we have from  \eqref{eqn:fourierhigh} that 
\[
\Vert \mu\Vert_{L^\infty(B_1)} \leq \rho^3 c_7\left( M e^{4Rt}+ \frac{\Vert \mu\Vert_{L^\infty(B_1)} }{t}\right)+c_1\frac{1}{\rho^{s-3}}.
\]
Taking $\rho = (\frac{t}{2C_7})^{\frac{1}{3}} $ leads to 
\[
\frac{1}{2}\Vert \mu\Vert_{L^\infty(B_1)} \leq \frac{M}{2} e^{4Rt}t+c_1 (2c_7)^{\frac{s-3}{3}} t^{-\frac{s-3}{3}}.
\]
Let $t = (1-\tau)\frac{\ln(3+M^{-1})}{4R}$,  $\tau\in(0,1)$, where $M$ is assumed to be sufficiently small such that  $t> t_0$. 

Combining the above estimates yields
\begin{align*}
\Vert \mu\Vert_{L^\infty(B_1)} &\leq \frac{M}{2}(3+M^{-1})^{1-\tau} \frac{\ln(3+M^{-1})}{4R}+c_1 (2c_7)^{\frac{s-3}{3}}\left(\frac{\ln(3+M^{-1})}{4R}\right)^{-\frac{s-3}{3}}\\
&\leq  \frac{(3M+1)^{1-\tau}}{8R}M^\tau\ln(3+M^{-1})+c_1 (8Rc_7)^{\frac{s-3}{3}}\left(\ln(3+M^{-1})\right)^{-\frac{s-3}{3}},
\end{align*}
which, under the assumption that $M$ is sufficient small, implies 
\begin{equation*}
\Vert \mu\Vert_{L^\infty(B_1)} \leq   2c_1 (8Rc_7)^{\frac{s-3}{3}}(\ln(3+M^{-1}))^{-\frac{s-3}{3}}.
\end{equation*}
Taking $C_2 = 2c_1 (8Rc_7)^{\frac{s-3}{3}}$, we arrive at the estimate \eqref{est:inhomowhite}. 

Next, we analyze the dependence of the constant $C_2$ on the wavenumber $k$. Recall that in the proof of Lemma \ref{lem:Fourierlow}, we have $c_7 =  2 (2\pi)^{-3}  \max\left\{c_3,   2\sqrt{\frac{4\pi}{3}}c_6 +c_6^2\right\}$, where $c_3 = O(k^5)$ and $c_6=O(k^2)$ as $k\to\infty$. Since $c_1$ depends only on $s$, we conclude that  $C_2 = O(k^{5(\frac{s}{3}-1)})$ as $k\to\infty$.
\end{proof}

We mention that the H\"{o}lder-type stability in Theorem \ref{thm:white_homo} will reduce to the logarithmic-type estimate in Theorem \ref{thm:white_inhomo} when the assumption $k>k_0$ is removed in Theorem \ref{thm:white_homo}. We provide a brief outline of the proof for the estimate in the case of a homogeneous medium below.

Consider the complex-valued special solutions  $U_i(x) = e^{\ii  \xi_t^{(i)}\cdot x}$, $i=1,2$, to the Helmholtz equation $\Delta u+k^2 u =0$, where $\xi_t^{(i)}$ are defined in \eqref{eq:xit1}--\eqref{eq:xit2}.
 For all $\gamma\in\mathbb{R}^3$, it follows the It\^{o} isometry that
 \[
\hat{\mu}(\gamma) =  (2\pi)^{-3} | \E[\langle f, U_1\rangle \langle f, U_2\rangle]|.
\]
By Lemma \ref{lem:Ewhitehomo}, we may deduce for $|\gamma|\leq\rho$ that 
\begin{align*}
|\hat{\mu}(\gamma)| &\leq (2\pi)^{-3} c_2 M \Vert U_1\Vert_{L^2(B_{\hat R})} \Vert U_2\Vert_{L^2(B_{\hat R})} \\
&\leq (2\pi)^{-3} c_2 M e^{4R\rho}.
\end{align*}
Then it can be verified that 
\begin{align*}
\Vert \mu\Vert_{L^\infty(B_1)} &\leq \sup_{x\in B_1}\left|\int_{\mathbb{R}^3} e^{\ii \gamma\cdot x} \hat{\mu}(\gamma)d\gamma\right|\\
&\leq \int_{|\gamma|\leq \rho} |\hat{\mu}(\gamma)|d\gamma+ \int_{|\gamma|>\rho} |\hat{\mu}(\gamma)|d\gamma\\
&\leq \rho^3 c_7M e^{4R\rho}+c_1\rho^{3-s}.
\end{align*}
Letting  $\rho= (1-\tau)\frac{\ln(3+M^{-1})}{4R}$,  $\tau\in(0,1)$, we obtain 
\begin{eqnarray*}
\Vert \mu\Vert_{L^\infty(B_1)} \leq &&\left((1-\tau)\frac{\ln(3+M^{-1})}{4R}\right)^3\\
&&\times \left(c_7M^\tau (3M+1)^{1-\tau}+c_1\left((1-\tau)\frac{\ln(3+M^{-1})}{4R}\right)^{-s}\right),
\end{eqnarray*}
which, under the assumption that $M$ is sufficiently small, implies that 
\[
\Vert \mu\Vert_{L^\infty(B_1)}\leq C\left(\ln(3+M^{-1})\right)^{3-s}.
\]

\section{Conclusion}\label{sec:4}

In this paper, we have explored the stability of the inverse source problems for the stochastic Helmholtz equation driven by white noise. We have established the well-posedness of the direct source problem and examined the stability of the solutions to the inverse source problems. By utilizing the special solutions to the associated Helmholtz equations, we have obtained both H\"{o}lder-type and logarithmic-type stability results for the cases of homogeneous and inhomogeneous media, respectively. Importantly, we have provided explicit expressions for the dependence of the stability constants on the wave number, thereby offering valuable insights into the behavior of the estimates. These results enhance our understanding of the inverse source problems and have practical implications for applications in medical imaging, specifically for acoustic and electromagnetic waves.

There are several potential avenues for future research in the field of inverse random source problems. In this study, we focused on modeling the random source as a white noise. However, there is room for exploring more general cases where the random source is modeled as a generalized random field. In such scenarios, the applicability of the It\^{o} isometry may be limited, and additional terms may arise due to source correlation and interaction with the media. Investigating these aspects and developing appropriate mathematical frameworks are promising directions for future investigation. Furthermore, the methodology and techniques employed in this study can be extended to handle different types of wave equations. Exploring the applicability and effectiveness of our approach in the context of other wave equations is another interesting avenue for future exploration. 

We anticipate reporting our findings and advancements in these areas in upcoming articles, where we will delve deeper into the challenges posed by generalized random fields and the extension of our methodology to various wave equations.

\bibliographystyle{siamplain}

\end{document}